\documentclass[12pt]{article}
\usepackage[textsize=tiny]{todonotes}
\usepackage{caption}
\usepackage{amsmath}                            % (1)
\usepackage{amssymb}                            % (1)
\usepackage{amsthm}                             % (2)
\usepackage[pdfborder={0 0 0}]{hyperref}        % (4)
\usepackage{algorithm}
\usepackage{hyperref}
\usepackage[utf8]{inputenc}
\usepackage{enumitem}
\usepackage{centernot}
\usepackage{placeins}
\usepackage{amsthm}
\usepackage{amsmath}
\usepackage{mathrsfs}
\usepackage{amssymb}
\usepackage{scalefnt}
\usepackage{amsfonts}
\usepackage{xcolor}
\usepackage{color}
\usepackage{xspace}
\usepackage{microtype}
\usepackage[mathscr]{eucal}
\usepackage{float}
\usepackage{amsthm}

\DeclareMathOperator{\dcdot}{\cdot\cdot}
\usepackage{centernot}
\usepackage{tikz}
\usepackage{placeins}
\usepackage{amsthm}
\usepackage{amsmath}
\usepackage{amssymb}
\usepackage{scalefnt}
\usepackage{amsfonts}
\usepackage{color}
\usepackage{xspace}
\usepackage{microtype}
\usepackage[mathscr]{eucal}

\newtheorem{theorem}{Theorem}%[section]
\newtheorem{lemma} [theorem] {Lemma}
\newtheorem{proposition}[theorem]{Proposition}
\theoremstyle{theorem}

\newtheorem{corollary}[theorem] {Corollary}
\theoremstyle{definition}
\newtheorem{definition}{Definition}
\theoremstyle{theorem}
\newtheorem{thmx}{Theorem}

\usepackage{algorithm}
\usepackage{algpseudocode}
\def\BState{\State\hskip-\ALG@thistlm}
\makeatother
\newcounter{claimCount}
\newenvironment{claim}{\medskip

    \noindent\refstepcounter{claimCount}\textbf{Claim~\arabic{claimCount}.}}{

    \medskip}
\newenvironment{claimproof}{\noindent\textit{Proof of Claim~\arabic{claimCount}.}}{\hfill\ensuremath{\qedsymbol} \tiny{Claim~\arabic{claimCount}}

    \medskip}

\begin{document}
\title{Constructing Geometric Graphs of Cop Number Three}	

\author{
Seyyed Aliasghar Hosseini\\
\texttt{sahossei@sfu.ca}
\and
Masood Masjoody\\
\texttt{mmasjood@sfu.ca}
%\and
%Bojan Mohar\\
%\texttt{Mohar@sfu.ca}
\and
Ladislav Stacho\\
\texttt{lstacho@sfu.ca}
}

\maketitle

\begin{abstract} 
The game of cops and robbers is a pursuit game on graphs where a set of agents, called the cops try to get to the same position of another agent, called the robber. Cops and robbers has been studies on several classes of graphs including geometrically represented graphs. For example, it has been shown that string graphs, including geometric graphs, have cop number at most 15 \cite{gavenvciak2018cops}. On the other hand, little is known about geometric graphs of any cop number less than 15 and there is only one example of a geometric graph of cop number three that has as many as 1440 vertices \cite{Andrew}. In this paper we present a construction for subdividing planar graphs of maximum degree $\le 5$ into geometric planar graphs of at least the same cop number. Indeed, our construction shows that there are infinitely many planar geometric graphs of cop number three. We also present another construction that consists in clique substitutions alongside subdividing the edges in a planar graph of maximum degree $\le 9$, resulting in geometric, but not necessarily planar, graphs of at least the same cop number as the starting graphs. Finally, we present a geometric graph of cop number three with 440 vertices.
\end{abstract}
\section{Introduction}\label{sec:intro}
A {\em game of cops and robbers} is a pursuit game on graphs, or a class of graphs, in which a set of agents, called the {\em cops}, try to get to the same position as another agent, called the {\em robber}. Among several variants of such a game, our focus will be on the one introduced in \cite{Aigner}, which is played on finite undirected graphs. Hence, we shall simply refer to this variant as ``the" game of cops and robbers. Let $G$ be a simple undirected graph. Consider a finite set of cops and a robber. The game on $G$ goes as follows. At the beginning of the game (step 1) each cop will be positioned in a vertex of $G$ and then the robber will be positioned in some vertex of $G$. In each of the subsequent steps each agent either moves to a vertex adjacent to its current position or stays still, with the robber taking turn after all of the cops. The cops win in a step $i$ of the game if in that step one of the cops gets to the vertex where the robber is located. The minimum number of cops that are guaranteed to capture the robber on $G$ in a finite number of steps is called the {\em cop number}  of $G$ and denoted $C(G)$. Graph $G$ is said to be {\em $k$-copwin} ($k\in\mathbb{N}$) if $C(G)\le k$. Since the cop number of a graph is equal to the  sum of the cop numbers of its components, whenever the cop number of a graph $G$ is concerned  $G$ is considered to be connected, unless otherwise is stated. Among the class of graphs with a bounded cop number one can mention the class of trees, which can easily shown to have cop number one, and the class of planar graphs, which have cop number at most three:

\begin{theorem}[\cite{Aigner}]\label{thm: planar-3-cop-win}
For every planar graph $G$ one have $C(G)\le 3$. 
\end{theorem}
Our results in this paper concern a class of geometrically represented graphs, known as {\em geometric graphs}. A geometric graph is a (drawing of a) graph having a finite subset $V$ of the plane as its vertex set, and whose edge set consists of line-segments between all pairs of distinct points in $V$ with Euclidean distance less that or equal to a positive constant $r$, called the {\em parameter} of the geometric graph. As such, ahead of determining whether a given drawing of a graph is geometric we need to fix its prospective parameter $r$, in which case $r$ shall be referred to as {\em the parameter of the geometric graphs}. A path drawn in the plane as a geometric graph is called a {\em geometric path}. Geometric graphs constitute a proper subclass of {\em string graphs}, as the intersection graphs of strings (or curves) in the plane. It has been shown that $C(G)\le 15$ for every string graph $G$ \cite{gavenvciak2018cops}, but there is no geometric graph known to have the cop number $\ge 4$. On the other hand, in \cite{Andrew} the authors provide a geometric graph on 1440 vertices with cop number three. Indeed, they represent a planar graph with girth five and minimum degree 3 as a geometric graph and use the following fact which gives a lower bound for cop number of graphs having girth $\ge 5$:

\begin{proposition}[\cite{Aigner}]\label{prop: girth-cop-num.}
For a graph $G$ with minimum degree $\delta$ one has $C(G)\ge \delta$ provided the girth of $G$ is at least 5.
\end{proposition}

Our main results in this paper are as follows:

\begin{thmx}\label{theorem: geoplanar1}
Every planar graph $G$ with maximum degree $\Delta\le 5$ has a subdivision into a planar geometric graph with cop number $C(G)$ or $C(G)+1$.
\end{thmx}

\begin{thmx}\label{theorem: geoplanar2}
For every planar graph $G$ with maximum degree $\Delta\le 9$, there is a subdivision of $K(G)$ with cop number $\ge C(G)$.% graph obtain from $G$ by a number of clique substitutions that admits a subdivision into a geometric graph $S(G)$ satisfying $C(S(G)\ge C(G)$.
\end{thmx}

Note that either result can be used to construct geometric graphs of cop number three. The construction provided in the proof of Theorem \ref{theorem: geoplanar1} is based on obtaining a polygonal-curve embedding from a given straight-line embedding of $G$ and then subdividing the edge-curve equally such that with an appropriate parameter for the geometric graphs, the resulting embedding is geometric. The latter, for example, requires that no subdividing vertex on an edge-polygonal curve be adjacent to a vertex belonging to another edge-polygonal curve. The latter, in particular, requires the angle between any two segments incident with a vertex of $G$ be greater than $\pi/3$. We use the same idea together with the operation of clique substitution for the construction in the proof of Theorem \ref{theorem: geoplanar2}. 

\begin{definition}[Clique Substitution]\cite{CDM154}
Let $G(V,E)$ be a graph and $v\in V(G)$. The {\em clique substitution at} $v$ is the graph obtain from $G$ by replacing $v$ with a clique of size $|N_G(v)|$ and matching vertices in $N_G(v)$ with the vertices of that clique. The {\em clique substitution of} $G$, denoted $K(G)$, is the graph obtained from $G$ by performing clique substitutions at all vertices of $G$. We refer to a clique substituted for a vertex of $G$ as a {\em knot} (of $K(G)$).
\end{definition}
\noindent Finally, utilizing the method of proof of Theorems \ref{theorem: geoplanar1} and \ref{theorem: geoplanar2}, we provide a representation of a graph on 450 vertices and cop number three as a geometric graph, which is substantially smaller than the 1440-vertex geometric graph of cop number three presented in \cite{Andrew}. More specifically, we show the following:

\begin{thmx}\label{thm: dodec }
The graph $G$ obtained by subdividing every edge of the dodecahedron into 15 edges has got cop number 3; moreover, it admits a geometric representation.
\end{thmx}

% the authors provide a geometric graph on 1440 vertices with cop number three. Indeed, they represent a planar graph with girth five and minimum degree 3 as a geometric graph. That such a graph has cop number three simply follows from Theorem \ref{thm: girth-cop-num.}. In Section \ref{sec: 450-vertex-3copwin} we improve this result by providing a representation of a graph on 450 vertices and cops number three as a geometric graph (Theorem \ref{thm: dodec }).

\section{Proof of Theorem \ref{theorem: geoplanar1}}\label{sec: planageom1}	

Given a planar graph $G$ with $\Delta(G)\le 5$, we first identify it with any of its straight-line embeddings, which exist according to F{\'a}ry's Theorem \cite{Fary}. Then, if necessary, we replace the endings of edge-segments with a polygonal curve of at most five segments in such a way that in the resulting polygonal-curve planar graph the angle between any pair of consecutive edges at a vertex is greater than $\pi/3$ and in every edge-curve, both of the angles between any two consecutive segments are also greater than $\pi/3$. Finally, we shall show that in such a polygonal-edge embedding of $G$ all edge-curves can be subdivided into paths of some fixed length so that the resulting graph is geometric. Thus, we can use the following lemma to establish Theorem \ref{theorem: geoplanar1}.

\begin{lemma}\label{lemma: cop-subdivision}\cite{CDM154}
Let $G'$ be the subdivision of a graph $G$ obtained by replacing every edge of $G$ with a path of length $l$ for some fixed $l\in\mathbb{N}$. Then,
\begin{equation}
C(G')\in\{C(G),C(G)+1\}. 
\end{equation}
\end{lemma}

\begin{lemma}\label{lemma: geo-planar}Let $A$ and $B$ be two points in the plane having Euclidean distance 1, and let $S$ be the square having  $AB$ as a diagonal. Then, given $k\in\mathbb{N}$, the parameter $r$ of geometric graphs can be set so that for each integer $l$ between  $5k+1$ and $4k^2+6k+1$ there exists a geometric path (i.e. a path drawn as a geometric graph) of length $l$ between $A$ and $B$ having no vertex outside of $S$.  
\end{lemma}
\begin{proof}\setcounter{claimCount}{0}Consider the Cartesian coordinate system where $A=(0,0)$ and $B=(1,0)$, and let $C=(1/2,1/2)$ and $D=(1/2,-1/2)$ be the other corners of $S$. %\textcolor{red}{Given $k\in\mathbb{N}$ pick $0<\alpha<\pi/4$ close enough to zero such that
%\begin{equation}\label{lemma: geo-planar01}
%\sin{\alpha}<\frac{1}{2(2k+1)^2}\quad \text{and}\quad \cos{\alpha}>\frac{4k+1}{4k+2},
%\end{equation}}
Given $k\in\mathbb{N}$ let
\begin{equation}\label{lemma: geo-planar01}
\alpha_k=\sin^{-1}\left( \frac{1}{2(2k+1)^2}\right) \qquad (0<\alpha_k<\pi/2).
\end{equation}
Observe that since sum of the squares of $(4k+1)/(4k+2) $ and $1/(2(2k+1)^2)$ is less than one, we have $\cos{\alpha_k}>(4k+1)/(4k+2)$ and, hence, $(2k+1)\tan{\alpha_k}<1/(4k+1)$. We let the parameter of geometric graphs be $r_k$ given by
\begin{equation}\label{lemma: geo-planar02}
r_k=\frac{\sqrt{2}}{4k}\left(1-(2k+1)\tan{\alpha_k} \right).
\end{equation}
As such, we will have
\begin{equation}\label{lemma: geo-planar02+}
r_k>\frac{\sqrt{2}}{4k+1}.
\end{equation}
Moreover, we set the vectors
\begin{alignat}{3}\label{lemma: geo-planar03}
\vec{X_k}&=\langle \cos{(\pi/4-\alpha_k)},\sin(\pi/4-\alpha_k)\rangle,\\ %\label{lemma: geo-planar04}
\vec{Y_k}&=-\langle \cos{(\pi/4+\alpha_k)},\sin(\pi/4+\alpha_k)\rangle,
\end{alignat}
\begin{figure}[h]
\begin{center}
 \includegraphics[scale=0.32]{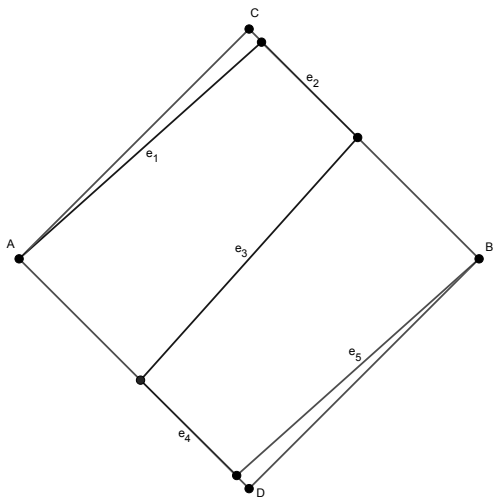}
 \caption{The polygonal curve $\gamma_1$ consisting of five line segments: $e_1$ and $e_5$ parallel to $X_1$, $e_3$ parallel to $Y_1$, and $e_2$ and $e_4$ of length $r_1$.}
 \label{pic:geoplan-first}
\end{center}
\end{figure}
\begin{figure}[h]
\begin{center}
 \includegraphics[scale=0.32]{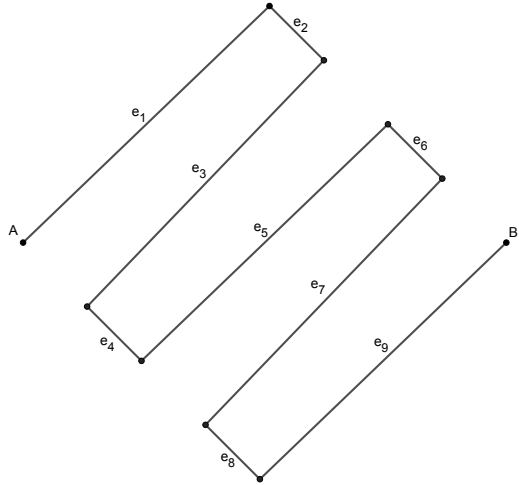}
 \caption{The polygonal curve $\gamma_2$ consisting of nine line segments: $X_2$-segments $e_1$, $e_5$, and $e_9$, $Y_2$-segments $e_3$ and $e_7$, and flat-segments $e_2$, $e_4$, $e_6$ and $e_8$ of length $r_2$.}
 \label{pic:geoplan-second}
\end{center}
\end{figure}
and let $\gamma_k$  be the polygonal curve from $A$ to $B$ consisting of $k+1$ line-segments parallel to $\vec{X_k}$, called {\em $\vec{X_k}$-segments}, each from a point on $AD$ to a point on $BC$, $k$ line-segments parallel to $\vec{Y_k}$, called {\em $\vec{Y_k}$-segments}, each from a point on $BC$ to a point on $AD$,  and $2k$ line-segments of length $r_k$ parallel to the directed line-segment from $B$ to $C$, called {\em flat segments}, such that the initial and terminal segments of $\gamma_k$ are $\vec{X_k}$-segments, and each of the first $k$ $\vec{X_k}$-segments in $\gamma_k$ is followed by exactly one flat segment which itself is followed by a $\vec{Y_k}$-segment, and, likewise, each of the $Y_k$-segments in $\gamma_k$ is followed by exactly one flat segment which is itself followed by an $\vec{X_k}$-segment- see Figures \ref{pic:geoplan-first} and \ref{pic:geoplan-second} for examples. We call any of the $k$ three-segment subcurves of $\gamma_k$ starting with an $X_k$-segment a {\em dent} of $\gamma_k$. By \eqref{lemma: geo-planar02}, we have $2rk<\sqrt{2}/{2}$. Moreover, with $l$ being the common length of $X_k$-segments and $Y_k$-segments (which can reasonably be referred to as {\em slant} segments) we have $l=\sqrt{2}/(2\cos{\alpha_k})<(2k+1)\sqrt{2}/(4k+1)$. Hence, according to \eqref{lemma: geo-planar02+}, we obtain
\begin{equation}\label{lemma: geo-planar05}
2r_k k< \frac{\sqrt{2}}{2}<l<(2k+1)r_k.
\end{equation}
Consequently, with $r_k$ as the parameter of geometric graphs, each of the slant segments of $\gamma_k$ can be subdivided into a geometric path of length $2k+1$. Thereby, as $\gamma_k$ consists of $2k$ flat segments of length $r_k$ and $2k+1$ slant segments, it can be subdivided into a geometric path of length $(2k+1)^2+2k=4k^2+6k+1$. We shall denote such a geometric path also by $\gamma_k$. Hence, to complete the proof it suffices to show the following:

\begin{claim} Given any dent $D$ of $\gamma_k$ and for each $s\in[2\dcdot (4k+2)]$, one can delete some vertices of $D$ and then introduce one or two new vertices inside $D$ to make $D$ into a geometric path of length $s$. Moreover, such a change can be made into each of the dents of $\gamma_k$ so that the entire resulting path will stay geometric. In particular, any integer between $5k+1$ and $4k^2+6k+1$ can be attained as the length of a geometric path between $A$ and $B$ having no vertex outside of $S$. 
\end{claim}
\begin{claimproof} Let $p_0,\dots,p_{2k+1},q_{2k+1},q_{2k},\dots,q_0$ be the sequence of vertices in $D$ and let $s=2t+1$ or $s=2t+2$ for some $t\in[0\dcdot k]$. Furthermore, let $w$ be the intersection of the segments $p_tq_{t+1}$ and $q_t p_{t+1}$ and pick points $p'$ and $q'$ on the segments $p_{t+1}w$ and  $q_{t+1}w$ such that $d_E(p_t,q')=d_E(q_t,p')> r_k$ and $d_E(p',q')\le r_k$ (Fig. \ref{pic:geoplan01}). Observe that the paths with sequence of vertices $p_0,\dots,p_t,w,q_t,\dots,q_0$ and $p_0,\dots,p_t,p',q',q_t,\dots,q_0$ are geometric. Hence, to obtain the desired length $s$ it suffices to remove all vertices $p_i,q_i$ with $i>t$ from $D$, and according as $s=2t+1$ or $s=2t+2$, add $w$ or both $p'$ and $q'$. Note that the newly added vertices will be at a distance greater than $r_k$ from vertices of $\gamma_k$ which are not in $D$. Hence, all of the dents of $\gamma_k$ can be adjusted this way while keeping the resulting path geometric. Since the last $Y_k$-segment of $\gamma_k$ and $k$ of the flat segments of $\gamma_k$ are in no dent and any dent can be reduced to a geometric path of any length between $2$ and $4k+2$ or kept at the original length of $4k+3$, $\gamma_k$ can be adjusted to a geometric path of any length between $2k+(2k+1)+k=5k+1$ and $4k^2+6k+1$.
\begin{figure}[H]
\begin{center}
 \includegraphics[scale=0.75]{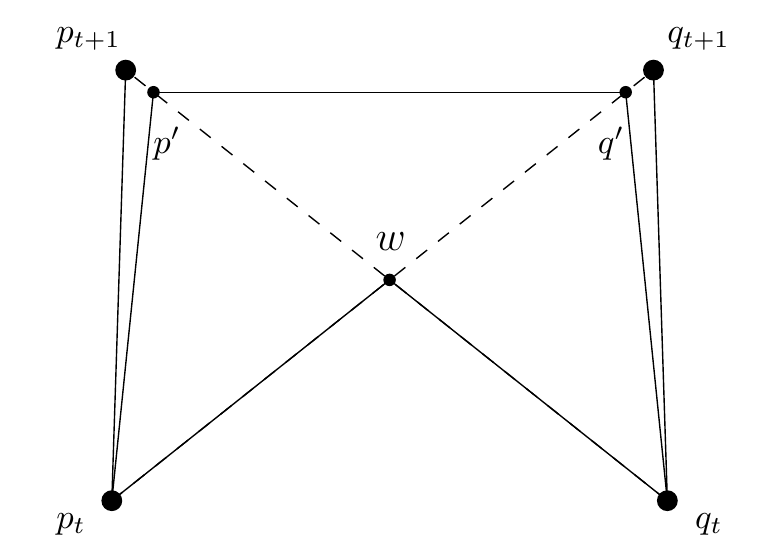}
 \caption{Adjusting the length of a dent}
 \label{pic:geoplan01}
\end{center}
\end{figure}
\end{claimproof}
\end{proof}
 %     flat segments  with and according as $s=2t$ or $s=

%, $and choose  circles of radius $r$ with centers $ Consider the arcs $C_1$+1} or $ \in[2\dcdot (4k+2)]$ is even, by deleting all vertices $p_i$, $q_i$ with $i\in[s/2\dcdot (2k+1)]$ and rotating $p_{s/2}$ and $q_{s/2}$ 
%\textcolor{red}{add the proof with some pictures of a modified dent.\\[300pt]}

%Let $P$ be such a geometric path.
 %On the other hand, one can subdivide $AC$into $\lceil 1/r \rceil=\lceil 2k\sqrt{2} \rceil$. In the rest of the proof we shall show that any integer $\lambda$ between $\lceil 2k\sqrt{2} \rceil$ and 
%$4k^2+6k+1$ can also be attained as the length of a geometric path between $A$ and $B$ with no vertex outside $S$.
\
%\textcolor{red}{doable, but technical}\\[100pt]
%\lambda \gammaand $2k$ flat segments in $\gamma$, $\gamma$ can be subdivided into a geometric path of length  Moreover, as $\alpha>0$
%Choose $r$ as the threshold for geometric graphs. Being slightly longer than the side length of $S$, each of the $a-$ and $b$-segments can be subdivided into $2k+1$ sub-segments $2k$ of which are of length $r$ and one in the middle of the subdivided line-segment is shorter than $r$. %After making the subdivisions, we obtain a geometric path of (graph) length $(2k+1)(2k+2)$. One the other hand, the diagonal $AB$ of $S$ can be subdivided into a geometric path of (graph) length $\lceil 2k\sqrt{2}\rceil$.\\

%\textbf{Adjustment of the long curve to obtain a geometric path of any fixed $l$},

% remaining one,among which only one in the middle has its length less than $r$ (and is slightly longer than the side length of $S$ and can be subdivided into $2k+1$ slightly longer  .to $B$ describestarting by an upward segment parallel to $\vect{u}$  and extending to the o followed byconsisting of $k+1$ upward segments  parallel to $\vect{u}$ extending from one side of $S$ to its opposite side, and $2k$ downward segments parallel to $v$ also extending from one side of $S$ to its opposite side, together with $2k$ segments along  and $k$ segments anti-parallel to $\vect{v}$, each set going from a side of $S$ to its opposite side, and    the line segments from starting from $A$
 
%\end{proof}
\noindent The following two technical lemmas will serve to justify that the proposed alterations of the terminal parts of edges of $G$ in the proof of Theorem \ref{theorem: geoplanar1} will keep the graph geometric without adding too many segments.\\
%{\Large \textcolor{red}{Up To Here}}\\
\begin{lemma}\label{lemma: hexagon01}
Let $O$ be a vertex of degree at most five in a straight-line plane graph $G$. Let $\mathscr{H}$ be the collection of all sets $H$ of six distinct rays emanating from $O$ such that the angle between any consecutive pair of rays in $H$ is $\pi/3$. For every $H\in\mathscr{H}$ let $\sigma(H)$ be the number of edges of $G$ incident with $O$ making an angle $\le\pi/37$ with a ray in $H$. Then $\min\{\sigma(H): H\in\mathscr{H}\}=0$.

\end{lemma}
\begin{proof}
Let $x= \min\{\sigma(H): H\in\mathscr{H}\}$. Note that one can rotate the rays of any $H\in\mathscr{H}$ to obtain some $H'\in \mathscr{H}$ satisfying $\sigma(H')\le 5-\sigma(H)$; hence, $x\le 2$. Consider some $L\in\mathscr{H}$ with $\sigma(L)=x$, and set $\alpha=\pi/18$. %and $\beta=5\alpha=5\pi/ 18$.
Note that if an edge incident with $O$ makes an angle $\le \pi/37$ with a ray in $L$, then any of the 10 rotations of $H$ about $O$ by $\pm\alpha$, $\pm2\alpha, \pm3\alpha$, $\pm4 \alpha $, and $\pm5\alpha$ puts that edge in an angular distance greater than $\pi/37$ from any ray in $L$. According to this observation, 
\begin{itemize}
\item we have $x\not=2$, for otherwise rotating the rays of $L$ by $2\alpha$ would leave no more than one edge incident with $O$ within angular distance of $\pi/37$ to a ray of $H$, a contradiction; and
\item we also have $x\not= 1$, for otherwise if one applies five consecutive rotations of all rays in $L$ about $O$ by $\pi/18$, after each of the rotations at least one new edge has to be placed within angular distance of $\pi/37$ to a ray in $L$, requiring $deg(O)\ge 6$, a contradiction.
% and, hence, performing the last rotations leaves no edge incident with $O$ in such an angular distance to any ray in $L$.
% We establish this fact using the observation that if $e$ is an edge of $G$ incident with $O$ which is within an angular distance of $\pi/37$ to a ray in some $H\in \mathscr{H}$, rotating $H$ by any angle between $2\pi/37$ and $\pi/3-2\pi/37$ puts $e$ in an angular distance $>\pi/37$ to any ray in $H$. According to our observation, if $\sigma(L)=1$ and one applies five consecutive rotations of  $L$ about $O$ by $\alpha=\pi/18$, after each of the the rotations at least one new edge has to be placed within angular distance of $\pi/37$ to a ray in $L$, a contradiction with $deg_v(G)\le 5$.
%To see this suppose, on the contrary, that . Observe that . Therefore,  contradiction $deg_v(G)\le 5$.% hence, performing the last rotations leaves no edge incident with $O$ in such an angular distance to any ray in $L$.

%after each of five consecutive rotations of all rays in $L$ about $O$ by $\alpha$, one new edge must be put in an angular disby $\alpha$ all of the edges incident with $O$ would be put be in an angular distance greater than or equal to $\pi/36$ to any ray in $L$, again a contradiction. 
\end{itemize}
Hence, $x=0$.
\end{proof}

\begin{lemma}\label{lemma: hexagon02}
With the assumptions and notation of Lemma \ref{lemma: hexagon01}, let $L\in\mathscr{H}$ such that $\sigma(L)=0$; and let $H_1$ and $H_2$ be distinct regular hexagons with center $O$ and side length $l_1$ and $l_2$ (with $l_1>l_2$) and corners on the rays of $L$. Then for every edge $e$ incident with $O$ one can pick $\pi(e)\in\{1,2\}$ and replace the segment of $e$ inside $H_{\pi(e)}$ with a simple polygonal curve $\alpha_e$ consisting of at most four line segments: a line segment $C_e$ that connects $O$ to a point $Q_e$ on the boundary of $H_{\pi(e)}$, together with a connected portion of the boundary of $H_{\pi(e)}$ comprising at most three line-segments, in such a way that the following hold (see Figure \ref{pic:geoplan010} and Figure \ref{pic:geoplan011}):
\begin{enumerate}
\item Every segment of each $\alpha_e$ is longer than
\begin{equation}\label{lemma: hexagon02-01}
s_{\pi(e)}:=\left(\frac{1}{2}-\frac{\sqrt{3}}{2}\tan(31\pi/222)\right)l_{\pi(e)} (> 0.093 l_{\pi(e)}).
\end{equation}%\textcolor{red}{No endpoint of a segment of $\gamma_e$ has an angular distance $\le \pi/36$ from any ray in $L$.}NOTE:INSTEAD OF THIS WE NEED A LOWER BOUND FOR THE SIZE OF SEGMENTS: NOSEGMENT SHORTER THEN $\pi/36*...a\; factor.of.the shorter\; of\;side\; lengths of H_1\; and H_2$
\item After replacing the ending of each $e$ with $\alpha_e$, no two consecutive segments on the polygonal curve representing $e$ have an angle $\le \pi/3$.
\item For every pair $e,e'$ of distinct edges incident with $O$ both angles (clockwise and counterclockwise) between $C_e$ and $C_{e'}$ are greater than $\pi/3$. Moreover, the minimum distance between $\alpha_e$ and $\alpha_{e'}$ is at least $\min\{(\sqrt{3}/2)(l_1-l_2),l_2,(\sqrt{3}/2)s_1,\delta\}$ where $s_1$ is given by \eqref{lemma: hexagon02-01}, and $\delta$ is the minimum distance between $e_1,\dots,e_5$ outside the smaller hexagon $H_2$.
% and such angles can be equal to $\pi/3$ only where $\gamma_e$ passes through a corner of $H_{\pi(e)}$.
%\item One can introduce arbitrarily small changes in the position of all corners in $\gamma_e$s that originally coincide with a corner of $H_1$ or $H_2$, such that properties (a), (b), and (c) are preserved and there is no angle equal to $\pi/3$ between two consecutive segments on any $\gamma_e$.
\end{enumerate}
\end{lemma}
\begin{proof}
For each $i\in\{1,2\}$ let the $A_{k,i}$ ($k=1,\dots,6$) be the corners of $H_i$ , say, clockwise around $O$ such that for every $k$, $A_{k,1}$ and $A_{k,2}$ belong to the same ray of $L$. We shall establish the lemma in the following extreme cases; the other cases can be deals with in a similar fashion.
\begin{description}
\item[Case I:] $\deg(O)=5$ and all edges incident with $O$ are between two consecutive rays in $L$; in other words, all such edges intersect the same side of $H_j$ $(j=1,2)$.
\item[Case II:] $\deg(O)=5$ and no two consecutive rays in $L$ enclose two of the edges incident with $O$; in other words, no to edges incident with $O$ cross the same side of $H_j$ ($j=1,2$).
\end{description}
Let $e_1,\dots,e_5$ be the edges incident with $O$ in the clockwise order around $O$, and for each $i\in[1\dcdot 5]$ and $j\in[1\dcdot 2]$ let $C_{i,j}$ be the intersection of $e_i$ with the boundary of $H_j$. We also denote every $Q_{e_i}$ simply with $Q_i$.\\

\begin{figure}[h]
\begin{center}
 \includegraphics[scale=0.35]{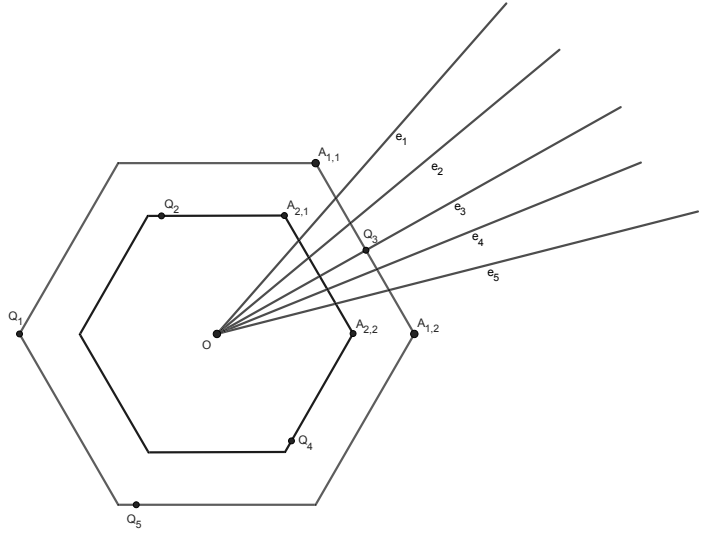}
 \caption{Case I: $\deg_G(O)=5$ and all edges incident with $O$ intersect the sides $A_{1,1}A_{1,2}$ of $H_1$ and $A_{2,1}A_{2,2}$ of $H_2$}
 \label{pic:geoplan010}
\end{center}
\end{figure}
\begin{figure}[h]
\begin{center}
 \includegraphics[scale=0.35]{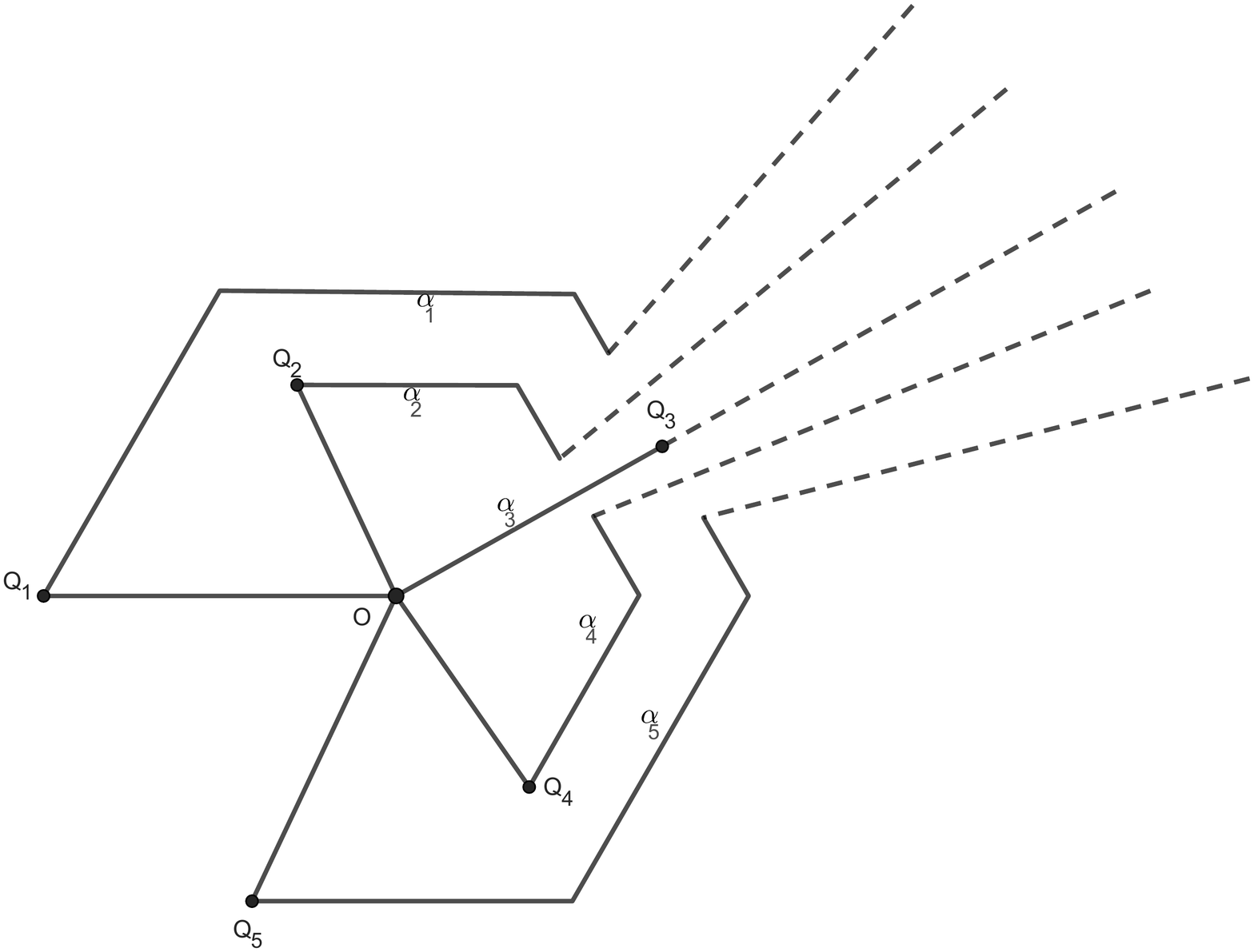}
 \caption{Replacing the endings in case I with polygonal curves $\alpha_1,\dots,\alpha_5$ to gain the desired properties}
 \label{pic:geoplan011}
\end{center}
\end{figure}

\noindent\textbf{Handling of Case I:} Suppose $e_1,\dots,e_5$ cross sides $A_{1,i}A_{2,i}$ of $H_i$ ($i=1,2$). % and suppose, without loss of generality, that $d_E(C_{3,1},A_{1,1})\le d_E(C_{3,1},A_{2,1})$. 
Set $\pi(e_1)=\pi(e_3)=\pi(e_5)=1$ and $\pi(e_2)=\pi(e_4)=2$. Also, set $Q_{e_1},\dots, Q_{e_5}$ as follows:
\begin{itemize}
\item $Q_{e_1}=A_{5,1}$%: the point on the segment $A_{4,1}A_{5,1}$ with $\measuredangle(Q_1,O,A_{5,1})=\pi/8)$;
\item $Q_{e_2}$: the point on the segment $A_{1,2}A_{6,2}$ with $\measuredangle(Q_2,O,A_{6,2})=\pi/37$;
\item $Q_{e_3}=C_{3,1}$;
\item $Q_{e_4}$: the point on the segment $A_{2,2}A_{3,2}$ with $\measuredangle(Q_4,O,A_{3,2})=\pi/37$; and
\item $Q_{e_5}$: the point on the segment $A_{3,1}A_{4,1}$ with $\measuredangle(Q_5,O,A_{4,1})=\pi/38$.
\end{itemize}
Furthermore, let $\alpha_{e_1}$ and $\alpha_{e_2}$ be counterclockwise around $H_1$ and $H_2$ , and $\alpha_{e_4}$ and $\alpha_{e_5}$ be clockwise around $H_2$ and $H_1$. Then, one can easily check that properties \textbf{(a)}-\textbf{(c)} are satisfied by $\alpha_{e_i}$s.\\[5pt]
\textbf{Handling of Case II:} Suppose $C_{i,1}$ belongs to the side $A_{i,1}A_{i+1,1}$ of $H_1$ for each $i\in[1\dcdot 5]$. As $\sigma(H_1)=0$, we have $\measuredangle(C_{i,1},O,A_{i+1})>\pi/37$ or, equivalently, $d_E(A_{i,1},C_{i,1})>s_1$ for each $i\in[1\dcdot 5]$, where $s_1$ is in \eqref{lemma: hexagon02-01}. Let
\begin{equation*}
a=\min\{d_E(A_{i,1},C_{i,1})-s_1: i\in[1\dcdot 5]\}.
\end{equation*}
Furthermore, to satisfy \textbf{(a)}-\textbf{(c)}, for each $i\in[1\dcdot 5]$ let $\alpha_i$ be clockwise around $H_1$ and pick $Q_{e_i}$ on the segment $C_{i-1}A_{i-1}$ such that
\begin{equation*}
d_E(Q_i,A_{i,1})=\frac{a(i-1)}{5}.
\end{equation*}
%o the side $A_{i,1}A_{i+1,1}$ of $H_1$ and let
%\begin{equation*}
%d_E(A_{i,1},C_{i,1})=s+a_i,
%\end{equation*}
%where $s=(1/2-\sqrt{3}/2\tan(31\pi/222))l_1$.  Let $a:=\min\{a_i:i\in[1\dcdot 5]\}$. Note that $a>0$, since each $a_i$ is positive. For each $i\in[1\dcdot 5]$ pick $Q_{e_i}$ on the segment $C_{i,1} A_{i-1}$ such that
%\begin{equation*}
%d_E(Q_i,A_{i,1})=\frac{a\cdot i}{6}, 
%\end{equation*}
%and let each $\gamma_{e_i}$ be counterclockwise around $H_1$.
\end{proof}
\begin{proof}[Proof of Theorem \ref{theorem: geoplanar1}] Let $\alpha$ be the minimum angle between pairs of edges of $G$ with a common endpoint, and choose $a>0$ small enough so that 
\begin{itemize}
\item the (minimum) Euclidean distance between any pair of non-incident edges is greater than $3a$; and
\item for every $v\in V(G)$, the ball of radius $6a$ centered at $v$ does not contain any vertex in $V(G)\setminus\{v\}$. 
\end{itemize}
Fixing $a$ as such, we break up every edge-segment $e_i=uv$ of length, say, $l_i$ into three parts, an initial part of length $2a$ starting, say, at $u$, a middle part of length $\lambda_i a$ where $\lambda_i\in\lceil l_i/a\rceil-4a$, and a terminal part ending at $v$ which is (necessarily) of a length between $2a$ and $3a$. Suppose $\lambda_1\le \dots\le \lambda_m$ where $m$ is the number of edges of $G$. Next, we apply Lemma \ref{lemma: hexagon02} to each vertex of the graph and with regular hexagons of side lengths $a$ and $1.5a$ to replace each of the initial and terminal parts of edge-segments with polygonal curves of at most five segments, one segment outside the hexagon associated with the edge-ending and up to four more segments, according to Lemma \ref{lemma: hexagon02}. Note that for each resulting edge polygonal-curve, the endings will consist of at most 10 segments of a total length less than $15a$. Therefore, for each $k\in \mathbb{N}$, by using $r_k a$ as the parameter of geometric graphs one can replace the endings of each edge with a total of not more than $10+15\lceil 1/r_k\rceil$ segments, which is bounded above by $10+60 k$, according to \eqref{lemma: geo-planar02+}. We choose $k\in\mathbb{N}$ large enough so that
\begin{equation*}
\lambda_1 k^2\ge 10+60k,\qquad \lambda_1 (3k^2+6k+1)\ge \lambda_m(5k+1),
\end{equation*}
and
\begin{equation*}
r_k \le \min\{2\sin(\alpha/2), \sqrt{3}/2(0.093)\}.
\end{equation*}
Then according to Lemma \ref{lemma: geo-planar}, the middle part of every edge polygonal curve can be replaced with a geometric path of appropriate lengths (between  $(5k+1)\lambda_i$ and $4k^2+6k+1)\lambda_i$ for each part of an initial (Euclidean) length $\lambda_ia$, so that the resulting graph $G'$ is a graph obtained from $G$ by replacing edges with paths of the same (graph) length. Moreover, according to Lemmas \ref{lemma: geo-planar} and \ref{lemma: hexagon02} and our choices for $a$ and $k$, $G'$ is a geometric plane graph. Finally, $C(G')\in\{C(G),C(G)+1\}$ according to Lemma \ref{lemma: cop-subdivision}.
\end{proof}
\begin{corollary}\label{cor.theorem: geoplanar1}
There is an infinite family of geometric graphs of cop number three.
\end{corollary}
\begin{proof}[Sketch of proof] Let $G$ be any straight-line embedding of the dodecahedron (or any other planar graph of cop number three). a planar graph of cop number three drawing of  graph of cop number three. Applying the construction described in the proof of Theorem \ref{theorem: geoplanar1} gives a planar geometric graph $G'$ with a parameter $r$, chosen as in the proof of the theorem. Given any $k\in \mathbb{N}$ replace every edge $e$ of $G'$ with a path of $k$ equal-length segments, without changing the geometry of $e$. The resulting plane graph $G_k$ will be a geometric graph with parameter $r/k$ (by construction) and the cop number of three.
\end{proof}
	
\section{Proof of Theorem \ref{theorem: geoplanar2}}\label{sec: planageom2}	
Let $G$ be a planar graph such that $\Delta(G)\le 9$, identified with any of its straight-line embeddings in the plane. We shall show how to construct a subdivision of (a drawing of) $K(G)$ which is geometric and has cop number $\ge C(G)$, where the latter will be established using Lemma \ref{lemma: cop-subdivision} alongside the following result:%perform clique substitutions show that by performing clique substitutions, one can edges one can build a straight-line drawing of $S(G)$, also named $S(G)$, with the following properties:
%\begin{itemize}
%\item For every knot $K$ of $S(G)$, $E(S(G))\setminus E(K)$ is disjoint from the interior of $K$,   
%\item two edges of $S(G)$ intersect only if they belong to the same knot of $S(G)$, and
%\item in $S(G)$ the proper edges can be replaced with polygonal curves with the same odd number of segments so that the resulting straight-line drawing, which we refer to as $G^*$, is a geometric graph. 
%\end{itemize}
%Then, it is easy to show that $C(G^*)\in\{C(G),C(G)+1\}$, a fact that can be to established by Lemma \ref{} and the following result:

\begin{lemma}\label{lemma: cop-clique}{\cite{CDM154}}
The operation of clique substitution does not decrease the cop number. %inparticular, for every graph $H$ one has $C(H)\le C(S(H))$.
\end{lemma}
%\todo{\cite{CDM154} provides a hard-to-read proof for $C(G)\le C(K(G)$ while it suffices to mirror a cop-winning strategy on any graph $G'$ obtained from $G$ by just one clique substitution to a winning strategy on $G$.}

\begin{proof}[Proof of Theorem \ref{theorem: geoplanar2} (Sketch)] The techniques are similar to those in the proof of Theorem \ref{theorem: geoplanar1} and related lemmas. The construction is carried out in two main phases:

\noindent \textbf{Phase I}: At each vertex $v$ of $G$ we pick a partition of the plane into cones with apex $v$ and angle $2\pi/9$, such that one of the edges incident with $v$ lies on one of the rays of the cones. We also consider four regular 9-gones $\Pi_{i,v}$ ($i\in[1\dcdot 4]$) centered at $v$ corresponding to the chosen partition of the plane at $v$, with distinct side lengths $s_i$ (independent of $v$) such that $\min\{s_i\}$ is substantially less that the shortest edge in $G$. Then, for each edge incident with $v$, we replace its end at $v$ with a polygonal curve consisting of a portion of the boundary of one of the $9$-gons and a segment from $v$ that lies on one of the rays in the decomposition. This phase is implemented in a similar fashion to the adjustments of the edge endings in Lemma \ref{lemma: hexagon02}, except that with up to nine possible edges incident with $v$, one needs to use the boundaries of four, rather than two, polygons.\\

\begin{figure}[H]
\begin{center}
 \includegraphics[scale=0.35]{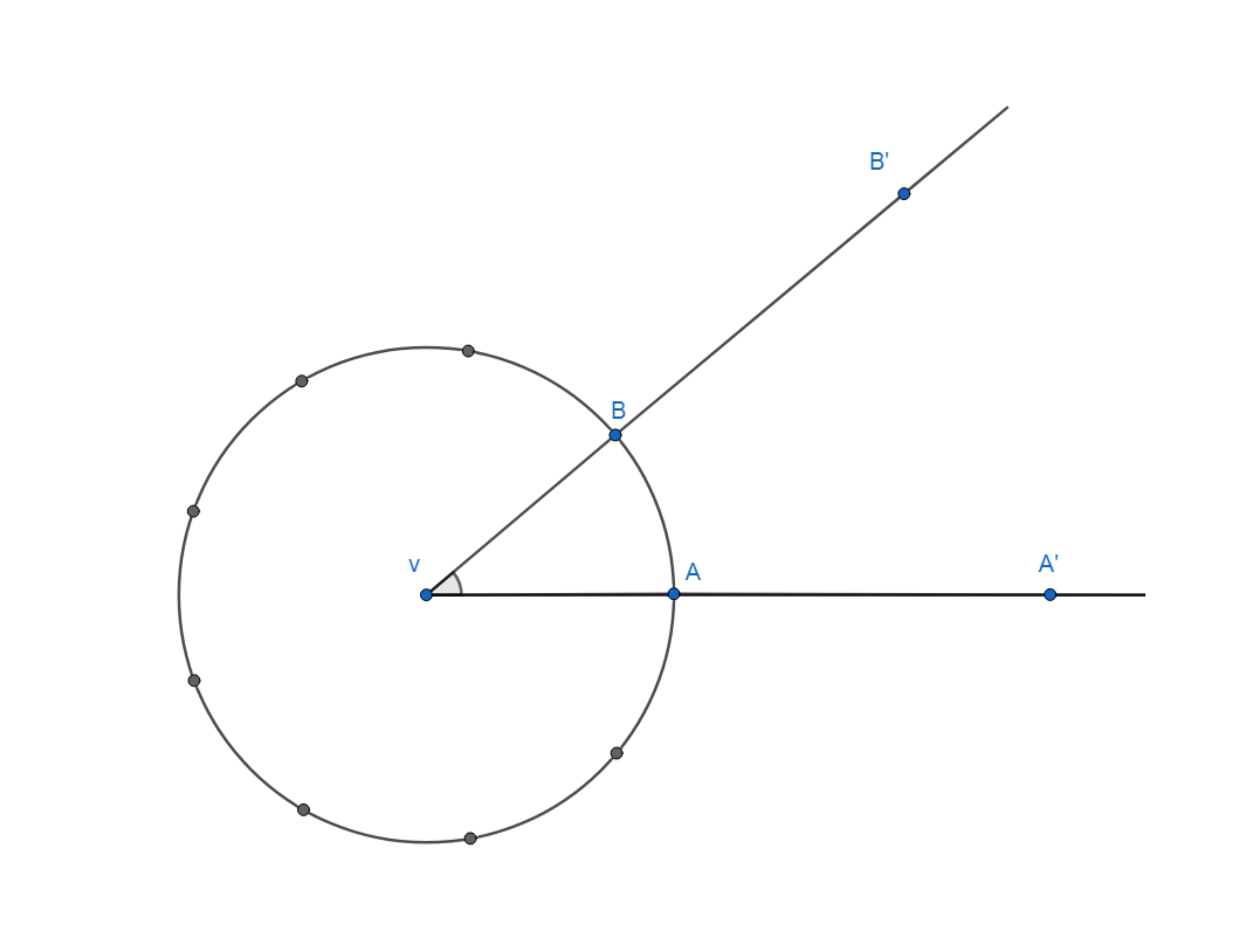}
 \caption{$A$ and $B$ are two consecutive vertices of a regular $n$-gon circumscribed by the circle of radius $r/2$ centered at $v$, and $|AA'|=|BB'|=r$. We need $|A'B'|>r$ or, equivalently, $n<\pi/(\sin^{-1}(1/3))$.} \label{pic:geoplan-last}
\end{center}
\end{figure}
\noindent \textbf{Phase II}: We pick the parameter $r$ of the geometric graphs such that $r\ll\min\{s_i\}$ and $r\ll\min_{i\not= j}|s_i-s_j|$. Then, we subdivide the edges so that the second vertices of endings at every vertex form a clique of size $deg_G(v)$. As such, by removing the original vertices of $G$, what we obtain is a drawing of a graph obtained from $K(G)$ by subdividing edges outside the knots. Finally, using the technique of Lemma \ref{lemma: geo-planar} we can adjust the latter graph to a geometric graph where the edges outsides the knots are subdivided into an equal number of edges. Note that we need $\Delta(G)\le 9$ in order to make sure that subdividing vertices for different edges incident to a vertex of $G$ do not lie within distance $r$ from each other (See Figure \ref{pic:geoplan-last}).  
\end{proof}
\section{Proof of Theorem \ref{thm: dodec }}\label{sec: planageom3}	
\begin{proof}[Proof of Theorem \ref{thm: dodec }]
Since the dodecahedron is a cubic graph with girth five, we have $C(G)=3$, according to Proposition \ref{prop: girth-cop-num.}; thereby, $C(G)\in\{3,4\}$, according to Lemma \ref{lemma: cop-subdivision}. But being a subdivision of a planar graph, $G$ is planar. Therefore, $C(G)=3$, according to Theorem \ref{thm: planar-3-cop-win}.
\begin{figure}[H]
\begin{center}
 \includegraphics[width=6in]{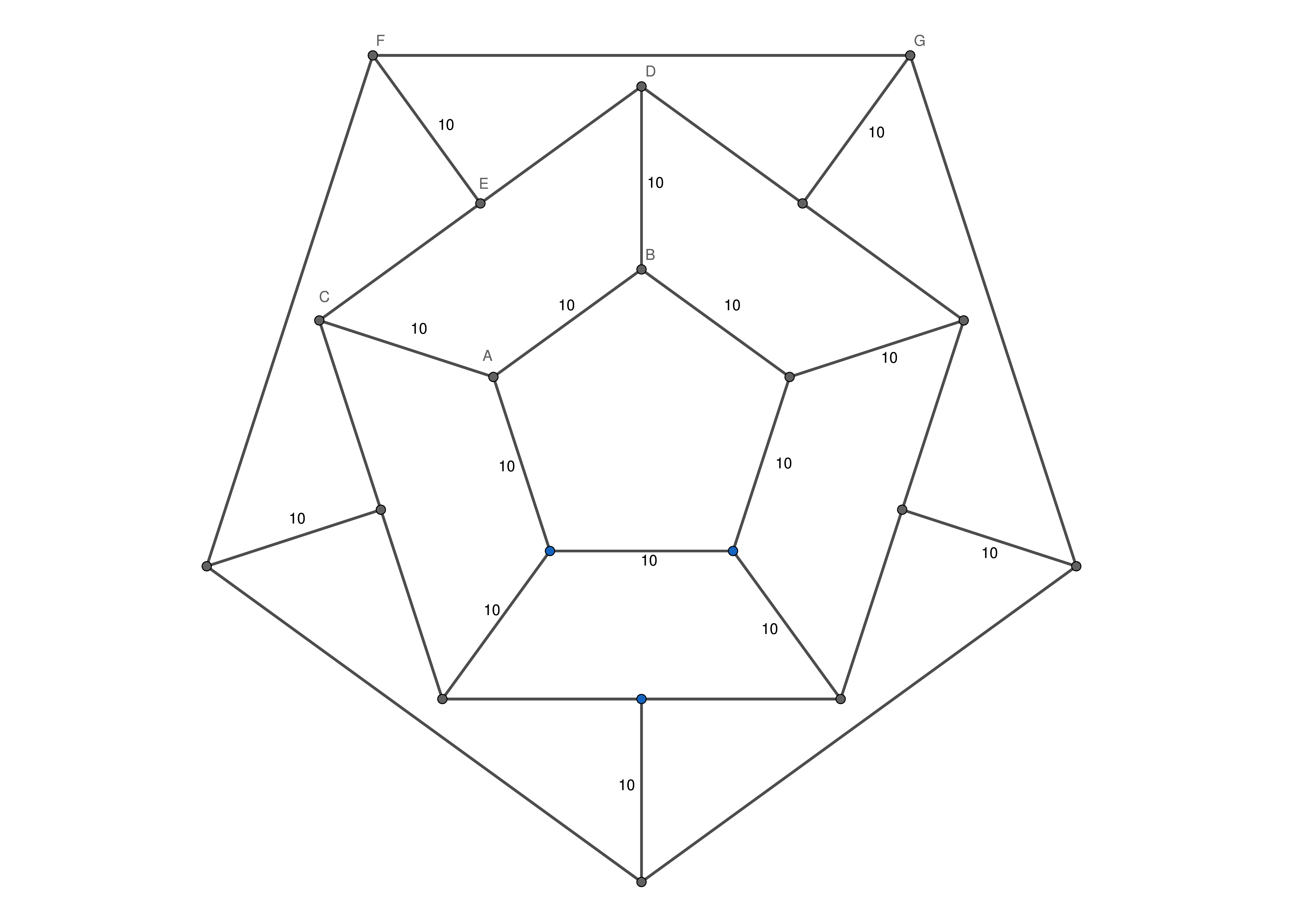}
 \caption{An embedding of the dodecahedron} \label{pic:dedec-1}
\end{center}
\end{figure}

 To complete the proof, we provide a geometric representation of $G$ derived from a specific straight-line embedding shown in Figure \ref{pic:dedec-1}. In this embedding, where number 10 next to some edges refers to their length, we consider five different type of edges represented by $AB$, $AC$, $CE$, $DE$, and $FG$. Next, we replace each of these edge types with appropriate polygonal curves, as shown in Figure \ref{pic:dedec-2}, so that the resultant graph remains planar
%\begin{itemize}
%\item The dodecahedron has cop number three.
%\item A straight-line embedding of the dodecahedron

%\item A polygonal-cure embedding of the Dodecahedron obtained from Figure \ref{pic:dedec-1}
\begin{figure}[H]
\begin{center}
 \includegraphics[width=5.3in]{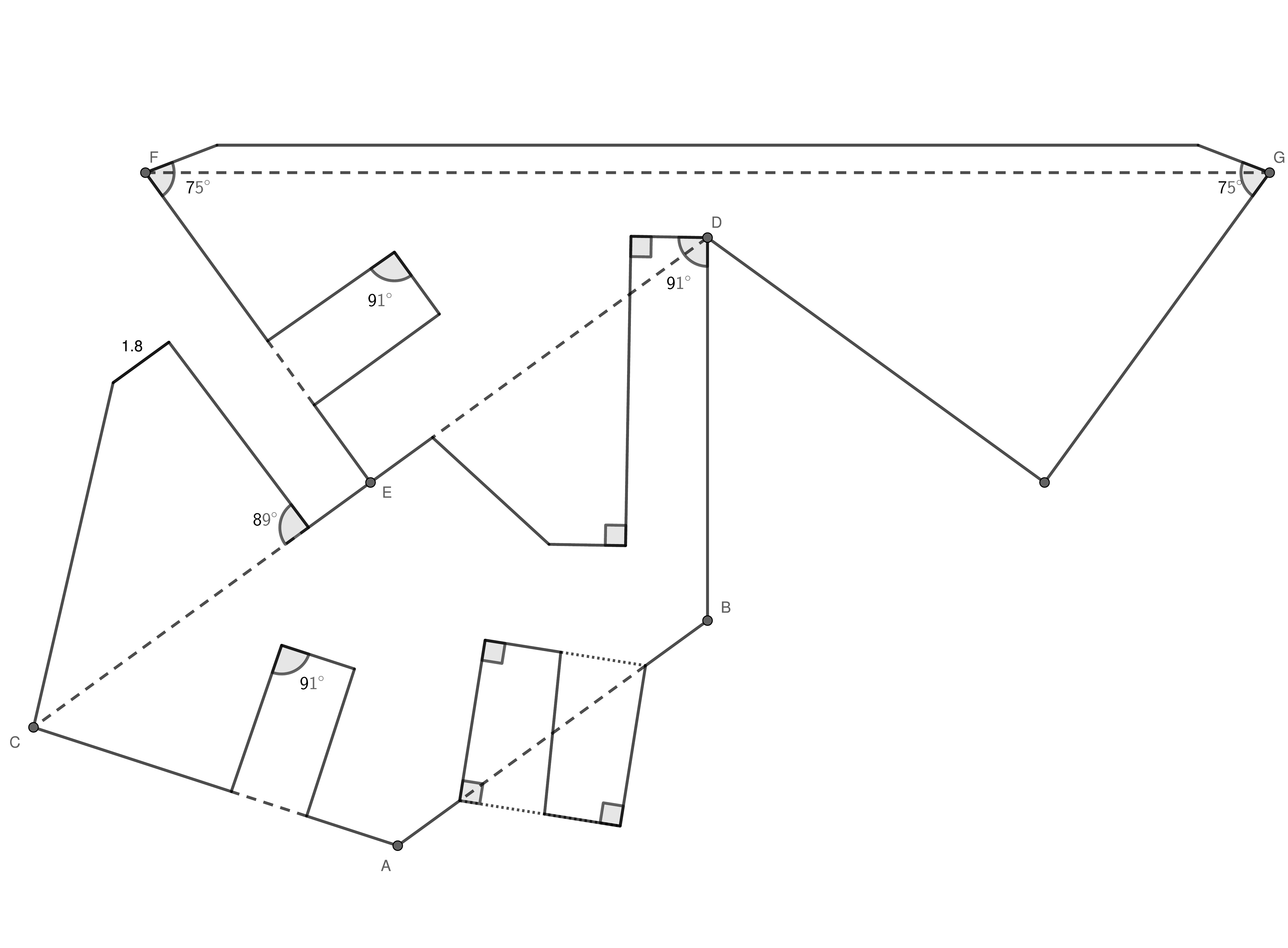}
 \caption{Replacing edges in Figure \ref{pic:dedec-1} with polygonal curves embedding of the Dodecahedron} \label{pic:dedec-2}
\end{center}
\end{figure}
Finally, we incorporate 14 new vertices along each of the polygonal curves curve in Figure \ref{pic:dedec-2} to made the embedding into a geometric graph. The details of the latter operations are shown in Figure \ref{pic:dedec-3}. One can easily check that the final embedding is indeed a geometric representation of $G$ with parameter $r=2$.
%\item Subdividing every edge in the polygonal-curve embedding obtained as in Figure \ref{pic:dedec-2} to 15 edges so that the resulting graph is geometric.
%\end{itemize}
\begin{figure}[H]
\centering
\begin{minipage}[t]{0.5\textwidth}
\centering
\includegraphics[width=3in]{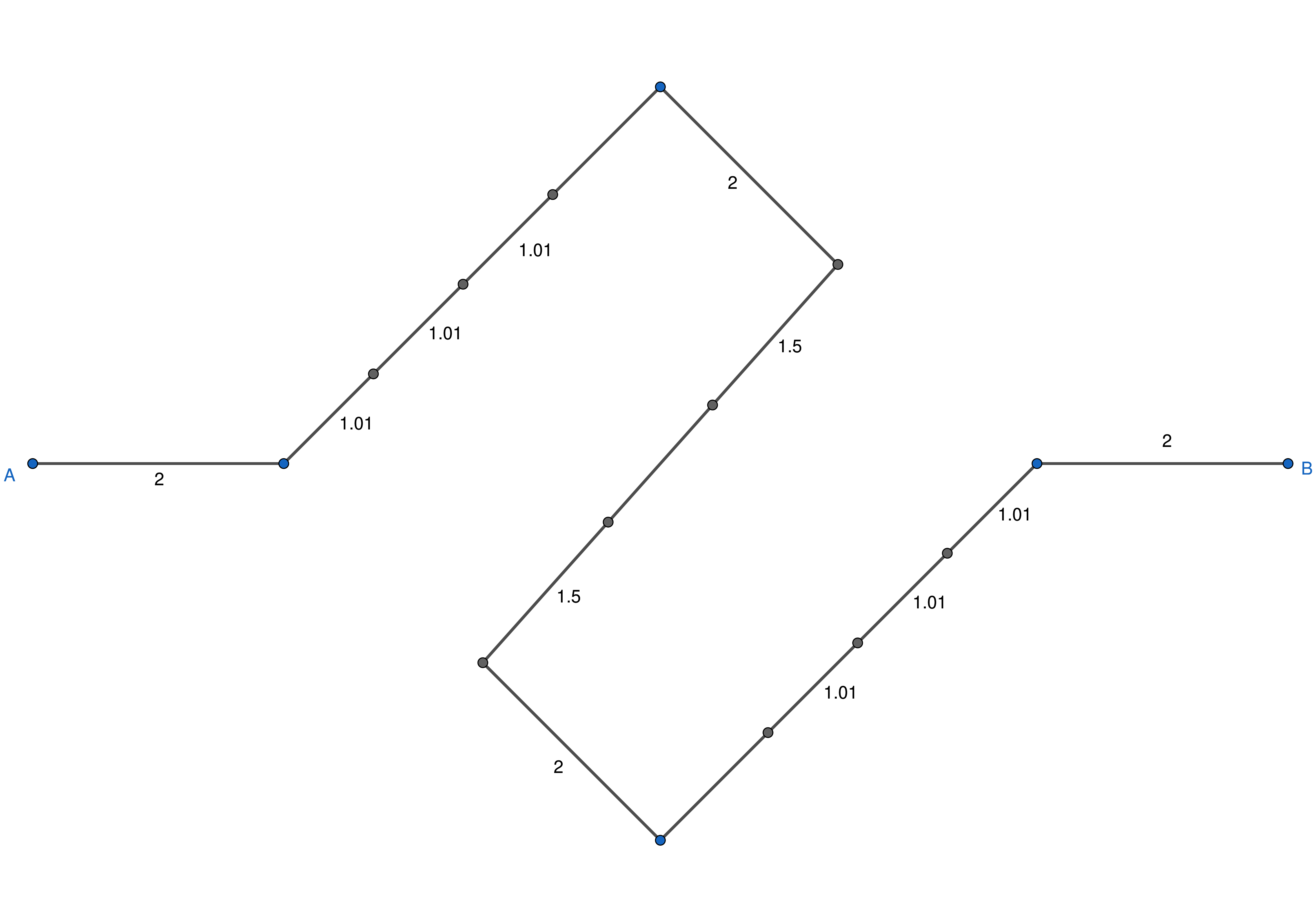}
\caption*{\textbf{(a) }Subdividing $AB$}
\end{minipage}\hfill\begin{minipage}[t]{0.5\textwidth}
\centering
\includegraphics[width=3in]{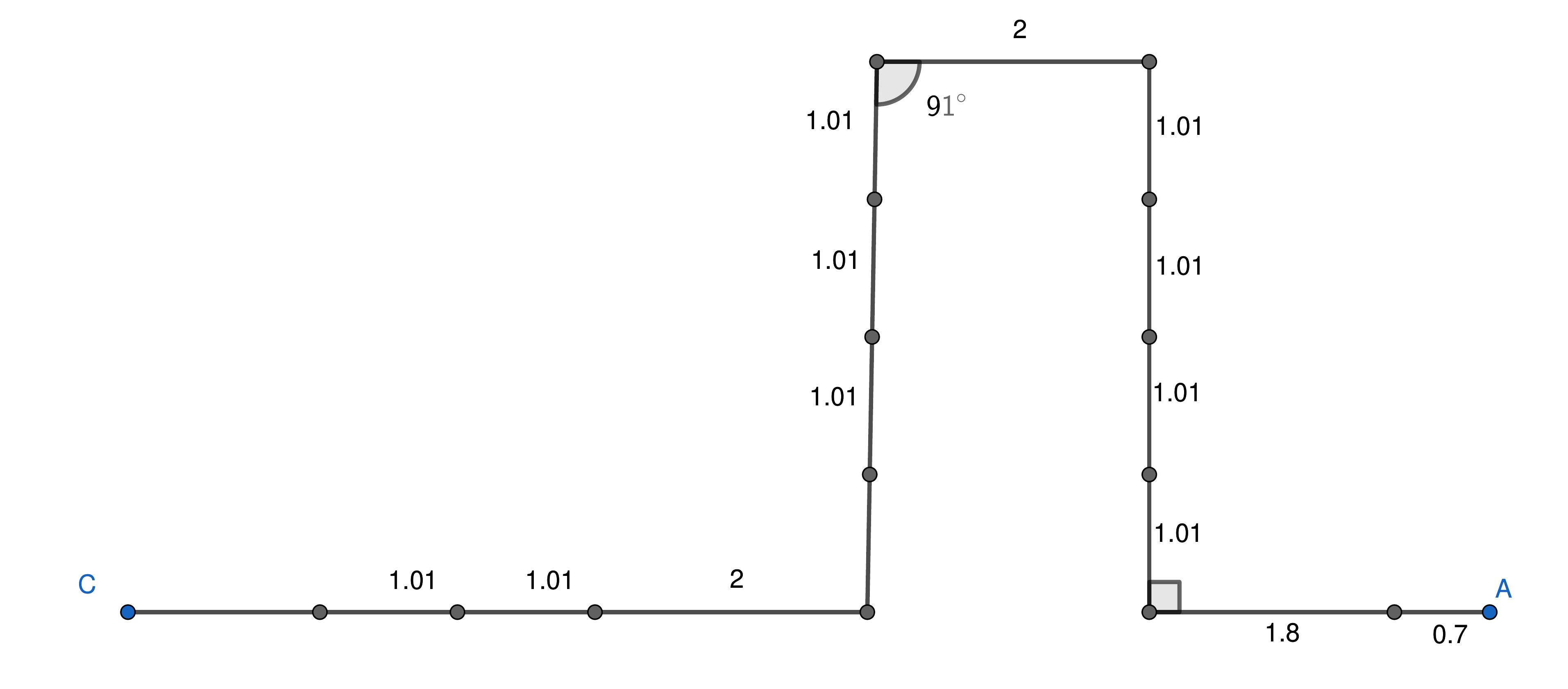}
\caption*{\textbf{(b) }Subdividing $AC$ and $EF$}
\end{minipage}
\begin{minipage}[t]{0.5\textwidth}
\centering
\includegraphics[width=3in]{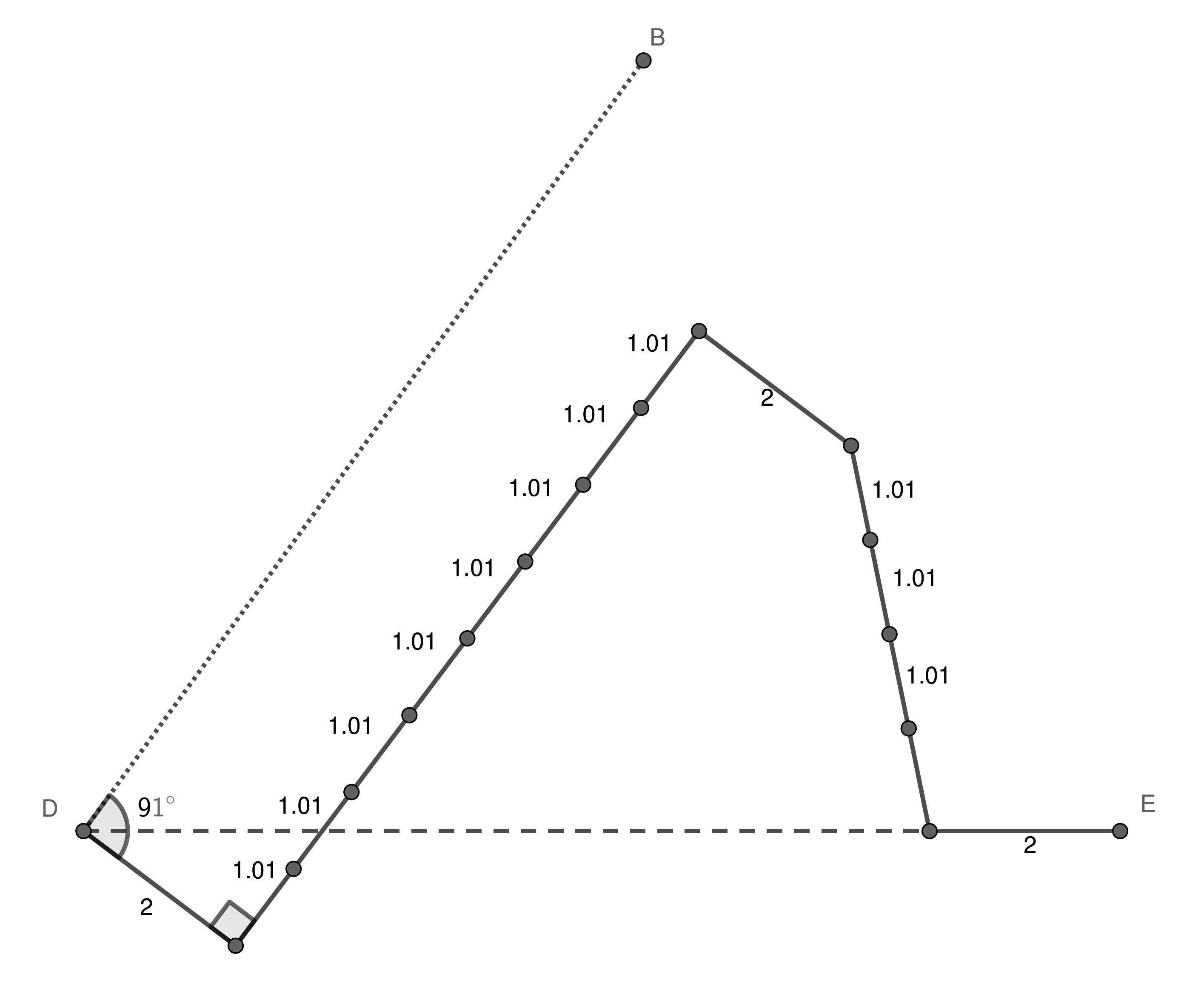}
\caption*{\textbf{(c) }Subdividing $DE$}
\end{minipage}\hfill\begin{minipage}[t]{0.5\textwidth}
\centering
\includegraphics[width=3in]{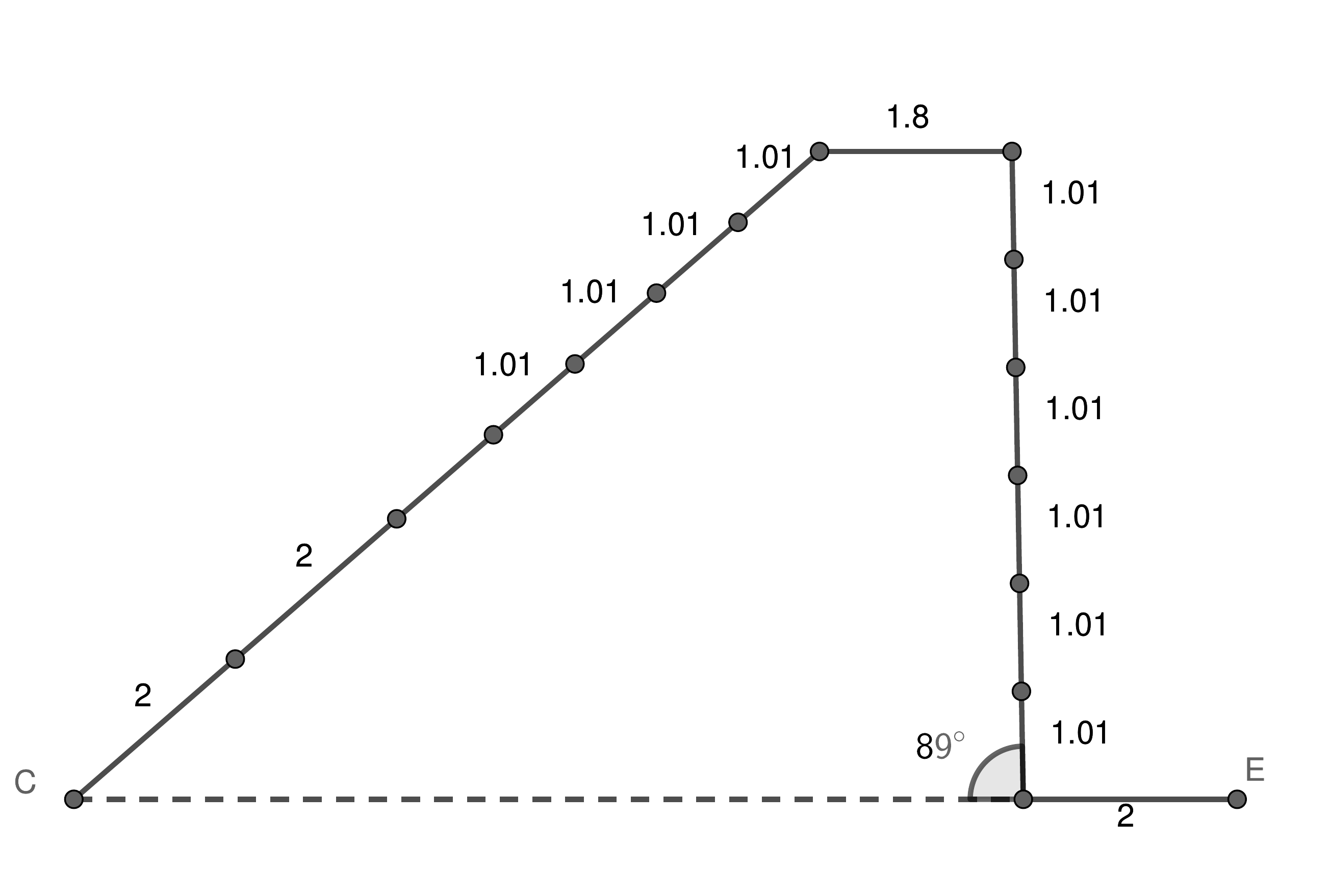}
\caption*{\textbf{(d) }Subdividing $CE$}
\end{minipage}\hfill\begin{minipage}[t]{0.9\textwidth}
\centering
\includegraphics[width=5.5in]{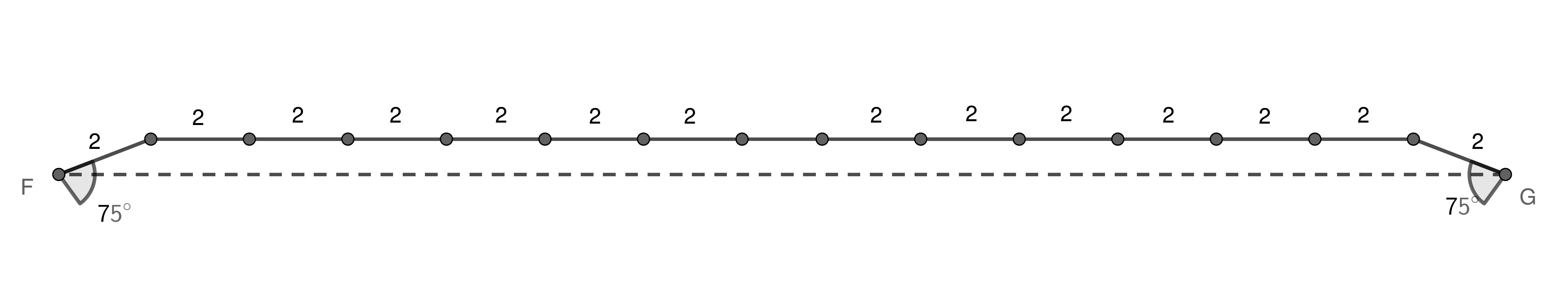}
\caption*{\textbf{(d) }Subdividing $FG$}
\end{minipage}\caption{Subdividing the graph obtain as in Figure \ref{pic:dedec-2}}\label{pic:dedec-3}
\end{figure}
\end{proof}
%\textcolor{red}{\section{A Small Geometric Graph of Cop Number Three}}\label{sec: small-geo-three }
%\textcolor{red}{\section{Concluding Remarks}}

\bibliographystyle{plain}

\end{document}